\documentclass[11pt,twoside]{article}
\usepackage{amsmath,amsthm,latexsym,amssymb,graphicx,hyperref}
\usepackage{datetime}
\usepackage{dsfont}
\usepackage{mathrsfs}
\usepackage{color}
\usepackage{enumitem}
\usepackage[all]{xy}

\theoremstyle{plain} 
\newtheorem{theorem}{Theorem}[section]
\newtheorem{lemma}[theorem]{Lemma}

\newtheorem{proposition}[theorem]{Proposition}

\theoremstyle{definition} 
\newtheorem{definition}[theorem]{Definition}

\newtheorem{example}[theorem]{Example}

\date{}

\begin{document}

\title{\bf The UP-Isomorphism Theorems for UP-algebras\footnote{This work was financially supported by the University of Phayao.}}

\author{Aiyared Iampan\footnote{Corresponding author.} \\[.3cm] Department of Mathematics, School of Science \\ University of Phayao, Phayao 56000, Thailand \\ Email: \texttt{aiyared.ia@up.ac.th}}\maketitle

\noindent\hrulefill

\begin{abstract}
In this paper, we construct the fundamental theorem of UP-homomorphisms in UP-algebras.
We also give an application of the theorem to the first, second, third and fourth UP-isomorphism theorems in UP-algebras.
\end{abstract}

\begin{flushleft}
\end{flushleft}

\noindent\hrulefill


\section{Introduction and Preliminaries}\numberwithin{equation}{section}


Among many algebraic structures, algebras of logic form important class of algebras.
Examples of these are BCK-algebras \cite{Imai1966}, BCI-algebras \cite{Iseki1966}, BCH-algebras \cite{Hu1983}, KU-algebras \cite{Prabpayak2009}, SU-algebras \cite{Keawrahun2011} and others.
They are strongly connected with logic. For example, BCI-algebras introduced by Is\'{e}ki \cite{Iseki1966} in 1966 have connections with BCI-logic being the BCI-system in combinatory logic which has application in the language of functional programming.
BCK and BCI-algebras are two classes of logical algebras.
They were introduced by Imai and Is\'{e}ki \cite{Imai1966,Iseki1966} in 1966 and have been extensively investigated by many researchers.
It is known that the class of BCK-algebras is a proper subclass of the class of BCI-algebras.


The isomorphism theorems play an important role in a general logical algebra, which were studied by several researches such as:
In 1998, Jun, Hong, Xin and Roh \cite{Jun1998} proved isomorphism theorems by using Chinese Remainder Theorem in BCI-algebras.
In 2001, Park, Shim and Roh \cite{Park2001} proved isomorphism theorems of IS-algebras.
In 2004, Hao and Li \cite{Hao2004} introduced the concept of ideals of an ideal in a BCI-algebra and some isomorphism theorems are obtained by using this concept. They obtained several isomorphism theorems of BG-algebras and related properties.
In 2006, Kim \cite{Kim2006} introduced the notion of KS-semigroups. He characterized ideals of a KS-semigroup and proved the
first isomorphism theorem for KS-semigroups.
In 2008, Kim and Kim \cite{Kim2008} introduced the notion of BG-algebras which is a generalization of B-algebras.
They obtained several isomorphism theorems of BG-algebras and related properties.
In 2009, Paradero-Vilela and Cawi \cite{Paradero2009} characterized KS-semigroup homomorphisms and proved the isomorphism theorems for KS-semigroups.
In 2011, Keawrahun and Leerawat \cite{Keawrahun2011} introduced the notion of SU-semigroups and proved the isomorphism theorems for SU-semigroups.
In 2012, Asawasamrit \cite{Asawasamrit2012} introduced the notion of KK-algebras and studied isomorphism theorems of KK-algebras.

Iampan \cite{Iampan2014} now introduced a new algebraic structure, called a UP-algebra and a concept of UP-ideals, congruences and UP-homomorphisms in UP-algebras, and defined a congruence relation on a UP-algebra and a quotient UP-algebra.
In this paper, we construct the fundamental theorem of UP-homomorphisms in UP-algebras.
We also give an application of the theorem to the first, second, third and fourth UP-isomorphism theorems in UP-algebras.

\medskip

Before we begin our study, we will introduce to the definition of a UP-algebra.


\begin{definition}\label{UP1}\cite{Iampan2014}
An algebra $A=(A,\cdot,0)$ of type $(2,0)$ is called a \textit{UP-algebra}, where $A$ is a nonempty set, $\cdot$ is a binary operation on $A$, and $0$ is a fixed element of $A$ (i.e., a nullary operation) if it satisfies the following axioms: for any $x,y,z\in A$,
\begin{description}
\item[(UP-1)] $(y\cdot z)\cdot((x\cdot y)\cdot(x\cdot z))=0$,
\item[(UP-2)] $0\cdot x=x$,
\item[(UP-3)] $x\cdot0=0$, and
\item[(UP-4)] $x\cdot y=y\cdot x=0$ implies $x=y$.
\end{description}
\end{definition}


\begin{example}\label{UP33}\cite{Iampan2014}
Let $X$ be a universal set. Define a binary operation $\cdot$ on the power set of $X$ by putting $A\cdot B=B\cap A^{\prime}=A^{\prime}\cap B=B-A$ for all $A,B\in \mathcal{P}(X)$.
Then $(\mathcal{P}(X),\cdot,\emptyset)$ is a UP-algebra and we shall call it the \textit{power UP-algebra of type 1}.
\end{example}


\begin{example}\label{DVT29}\cite{Iampan2014}
Let $X$ be a universal set. Define a binary operation $\ast$ on the power set of $X$ by putting $A\ast B=B\cup A^{\prime}=A^{\prime}\cup B$ for all $A,B\in \mathcal{P}(X)$.
Then $(\mathcal{P}(X),\ast,X)$ is a UP-algebra and we shall call it the \textit{power UP-algebra of type 2}.
\end{example}


\begin{example}\label{UP34}\cite{Iampan2014}
Let $A=\{0,a,b,c\}$ be a set with a binary operation $\cdot$ defined by the following Cayley table:
\begin{equation}\label{eq13}
\begin{array}{c|cccc}
  \cdot & 0 & a & b & c \\
  \hline
  0 & 0 & a & b & c \\
  a & 0 & 0 & 0 & 0 \\
  b & 0 & a & 0 & c \\
  c & 0 & a & b & 0
\end{array}
\end{equation}
Then $(A,\cdot,0)$ is a UP-algebra.
\end{example}


The following proposition is very important for the study of UP-algebras.

\begin{proposition}\label{UP2}\cite{Iampan2014}
In a UP-algebra $A$, the following properties hold: for any $x,y,z\in A$,
\begin{enumerate}[label=\textrm{(\arabic*)}]
\item\label{UP2_1} $x\cdot x=0$,
\item\label{UP2_13} $x\cdot y=0$ and $y\cdot z=0$ implies $x\cdot z=0$,
\item\label{UP2_6} $x\cdot y=0$ implies $(z\cdot x)\cdot (z\cdot y)=0$,
\item\label{UP2_7} $x\cdot y=0$ implies $(y\cdot z)\cdot (x\cdot z)=0$,
\item\label{UP2_14} $x\cdot(y\cdot x)=0$,
\item\label{UP2_16} $(y\cdot x)\cdot x=0$ if and only if $x=y\cdot x$, and
\item\label{UP2_15} $x\cdot(y\cdot y)=0$.
\end{enumerate}
\end{proposition}


\begin{theorem}\label{UP25}\cite{Iampan2014}
An algebra $A=(A,\cdot,0)$ of type $(2,0)$ is a UP-algebra if and only if it satisfies the following conditions: for all $x,y,z\in A$,
\begin{enumerate}[label=\textrm{(\arabic*)}]
\item\label{UP25_1} \emph{(UP-1):} $(y\cdot z)\cdot((x\cdot y)\cdot(x\cdot z))=0$,
\item\label{UP25_2} $(y\cdot0)\cdot x=x$, and
\item\label{UP25_3} \emph{(UP-4):} $x\cdot y=y\cdot x=0$ implies $x=y$.
\end{enumerate}
\end{theorem}


\begin{definition}\label{UP8}\cite{Iampan2014}
Let $A$ be a UP-algebra. A nonempty subset $B$ of $A$ is called a \textit{UP-ideal} of $A$ if it satisfies the following properties:
\begin{enumerate}[label=\textrm{(\arabic*)}]
\item\label{UP8_1} the constant $0$ of $A$ is in $B$, and
\item\label{UP8_2} for any $x,y,z\in A,x\cdot(y\cdot z)\in B$ and $y\in B$ implies $x\cdot z\in B$.
\end{enumerate}
Clearly, $A$ and $\{0\}$ are UP-ideals of $A$.
\end{definition}

We can easily show the following example.

\begin{example}\label{UP36}\cite{Iampan2014}
Let $A=\{0,a,b,c,d\}$ be a set with a binary operation $\cdot$ defined by the following Cayley table:
\begin{equation}\label{eq14}
\begin{array}{c|ccccc}
  \cdot & 0 & a & b & c & d \\
  \hline
  0 & 0 & a & b & c & d \\
  a & 0 & 0 & b & c & d \\
  b & 0 & 0 & 0 & c & d \\
  c & 0 & 0 & b & 0 & d \\
  d & 0 & 0 & 0 & 0 & 0
\end{array}
\end{equation}
Then $(A,\cdot,0)$ is a UP-algebra and $\{0,a,b\}$ and $\{0,a,c\}$ are UP-ideals of $A$.
\end{example}


\begin{theorem}\label{UP9}\cite{Iampan2014}
Let $A$ be a UP-algebra and $B$ a UP-ideal of $A$.
Then the following statements hold: for any $x,a,b\in A$,
\begin{enumerate}[label=\textrm{(\arabic*)}]
\item\label{UP9_1} if $b\cdot x\in B$ and $b\in B$, then $x\in B$. Moreover, if $b\cdot X\subseteq B$ and $b\in B$, then $X\subseteq B$,
\item\label{UP9_2} if $b\in B$, then $x\cdot b\in B$. Moreover, if $b\in B$, then $X\cdot b\subseteq B$, and
\item\label{UP9_3} if $a,b\in B$, then $(b\cdot(a\cdot x))\cdot x\in B$.
\end{enumerate}
\end{theorem}


\begin{theorem}\label{UP10}\cite{Iampan2014}
Let $A$ be a UP-algebra and $\{B_{i}\}_{i\in I}$ a family of UP-ideals of $A$.
Then $\bigcap_{i\in I}B_{i}$ is a UP-ideal of $A$.
\end{theorem}


From Theorem \ref{UP10}, the intersection of all UP-ideals of a UP-algebra $A$ containing a subset $X$ of $A$ is the UP-ideal of $A$ generated by $X$.
For $X=\{a\}$, let $I(a)$ denote the UP-ideal of $A$ generated by $\{a\}$.
We see that the UP-ideal of $A$ generated by $\emptyset$ and $\{0\}$ is $\{0\}$, and the UP-ideal of $A$ generated by $A$ is $A$.


\begin{definition}\label{UP11}\cite{Iampan2014}
Let $A=(A,\cdot,0)$ be a UP-algebra. A subset $S$ of $A$ is called a \textit{UP-subalgebra} of $A$ if the constant $0$ of $A$ is in $S$, and $(S,\cdot,0)$ itself forms a UP-algebra.
Clearly, $A$ and $\{0\}$ are UP-subalgebras of $A$.
\end{definition}

Applying Proposition \ref{UP2} \ref{UP2_1}, we can then easily prove the following Proposition.

\begin{proposition}\label{UP20}\cite{Iampan2014}
A nonempty subset $S$ of a UP-algebra $A=(A,\cdot,0)$ is a UP-subalgebra of $A$ if and only if $S$ is closed under the $\cdot$ multiplication on $A$.
\end{proposition}


\begin{theorem}\label{UP26}\cite{Iampan2014}
Let $A$ be a UP-algebra and $\{B_{i}\}_{i\in I}$ a family of UP-subalgebras of $A$.
Then $\bigcap_{i\in I}B_{i}$ is a UP-subalgebra of $A$.
\end{theorem}


\begin{theorem}\label{UP12}\cite{Iampan2014}
Let $A$ be a UP-algebra and $B$ a UP-ideal of $A$.
Then $A\cdot B\subseteq B$. In particular, $B$ is a UP-subalgebra of $A$.
\end{theorem}


\begin{theorem}\label{UP30}\cite{Iampan2014}
Let $A$ be a UP-algebra and $B$ a UP-subalgebra of $A$.
If $S$ is a subset of $B$ that is satisfies the following properties:
\begin{enumerate}[label=\textrm{(\arabic*)}]
\item\label{UP30_1} the constant $0$ of $A$ is in $S$, and
\item\label{UP30_2} for any $x,a,b\in B$, if $a,b\in S$, then $(b\cdot(a\cdot x))\cdot x\in S$.
\end{enumerate}
Then $S$ is a UP-ideal of $B$.
\end{theorem}


\begin{definition}\label{UP13}\cite{Iampan2014}
Let $A$ be a UP-algebra and $B$ a UP-ideal of $A$.
Define the binary relation $\sim_{B}$ on $A$ as follows: for all $x,y\in A$,
\begin{equation}\label{eq9}
x\sim_{B} y \textrm{ if and only if } x\cdot y\in B \textrm{ and } y\cdot x\in B.
\end{equation}
\end{definition}


\begin{definition}\label{UP15}\cite{Iampan2014}
Let $A$ be a UP-algebra. An equivalence relation $\rho$ on $A$ is called a \textit{congruence} if for any $x,y,z\in A$,
\begin{equation*}
x\rho y \textrm{ implies } x\cdot z\rho y\cdot z \textrm{ and } z\cdot x\rho z\cdot y.
\end{equation*}
\end{definition}


\begin{proposition}\label{UP14}\cite{Iampan2014}
Let $A$ be a UP-algebra and $B$ a UP-ideal of $A$ with a binary relation $\sim_{B}$ defined by \eqref{eq9}.
Then $\sim_{B}$ is a congruence on $A$.
\end{proposition}


Let $A$ be a UP-algebra and $\rho$ a congruence on $A$.
If $x\in A$, then the $\rho$-class of $x$ is the $(x)_{\rho}$ defined  as follows:
\begin{equation*}
(x)_{\rho}=\{y\in A\mid y\rho x\}.
\end{equation*}
Then the set of all $\rho$-classes is called the \textit{quotient set of $A$ by $\rho$}, and is denoted by $A/\rho$.
That is,
\begin{equation*}
A/\rho=\{(x)_{\rho}\mid x\in A\}.
\end{equation*}


\begin{theorem}\label{UP17}\cite{Iampan2014}
Let $A$ be a UP-algebra and $\rho$ a congruence on $A$.
Then the following statements hold:
\begin{enumerate}[label=\textrm{(\arabic*)}]
\item\label{UP17_1} the $\rho$-class $(0)_{\rho}$ is a UP-ideal and a UP-subalgebra of $A$,
\item\label{UP17_3} a $\rho$-class $(x)_{\rho}$ is a UP-ideal of $A$ if and only if $x\rho0$, and
\item\label{UP17_4} a $\rho$-class $(x)_{\rho}$ is a UP-subalgebra of $A$ if and only if $x\rho0$.
\end{enumerate}
\end{theorem}


\begin{theorem}\label{UP19}\cite{Iampan2014}
Let $A$ be a UP-algebra and $B$ a UP-ideal of $A$.
Then the following statements hold:
\begin{enumerate}[label=\textrm{(\arabic*)}]
\item\label{UP19_1} the $\sim_{B}$-class $(0)_{\sim_{B}}$ is a UP-ideal and a UP-subalgebra of $A$ which $B=(0)_{\sim_{B}}$,
\item\label{UP19_3} a $\sim_{B}$-class $(x)_{\sim_{B}}$ is a UP-ideal of $A$ if and only if $x\in B$,
\item\label{UP19_4} a $\sim_{B}$-class $(x)_{\sim_{B}}$ is a UP-subalgebra of $A$ if and only if $x\in B$, and
\item\label{UP19_2} $(A/\sim_{B},\ast,(0)_{\sim_{B}})$ is a UP-algebra under the $\ast$ multiplication defined by $(x)_{\sim_{B}}\ast(y)_{\sim_{B}}=(x\cdot y)_{\sim_{B}}$ for all $x,y\in A$, called the \textit{quotient UP-algebra} of $A$ induced by the congruence $\sim_{B}$.
\end{enumerate}
\end{theorem}


\begin{definition}\label{UP22}\cite{Iampan2014}
Let $(A,\cdot,0)$ and $(A^{\prime},\cdot^{\prime},0^{\prime})$ be UP-algebras.
A mapping $f$ from $A$ to $A^{\prime}$ is called a \textit{UP-homomorphism} if
\begin{equation*}
f(x\cdot y)=f(x)\cdot^{\prime}f(y) \textrm{ for all } x,y\in A.
\end{equation*}

A UP-homomorphism $f\colon A\to A^{\prime}$ is called a
\begin{enumerate}[label=\textrm{(\arabic*)}]
\item\label{UP22_1} \textit{UP-epimorphism} if $f$ is surjective,
\item\label{UP22_2} \textit{UP-monomorphism} if $f$ is injective,
\item\label{UP22_3} \textit{UP-isomorphism} if $f$ is bijective. Moreover, we say $A$ is \textit{UP-isomorphic} to $A^{\prime}$, symbolically, $A\cong A^{\prime}$, if there is a UP-isomorphism from $A$ to $A^{\prime}$.
\end{enumerate}

Let $f$ be a mapping from $A$ to $A^{\prime}$, and let $B$ be a nonempty subset of $A$, and $B^{\prime}$ of $A^{\prime}$.
The set $\{f(x)\mid x\in B\}$ is called the \textit{image} of $B$ under $f$, denoted by $f(B)$.
In particular, $f(A)$ is called the \textit{image} of $f$, denoted by $\mathrm{Im}(f)$.
Dually, the set $\{x\in A\mid f(x)\in B^{\prime}\}$ is said the \textit{inverse image} of $B^{\prime}$ under $f$, symbolically, $f^{-1}(B^{\prime})$.
Especially, we say $f^{-1}(\{0^{\prime}\})$ is the \textit{kernel} of $f$, written by $\mathrm{Ker}(f)$. That is,
\begin{equation*}
\mathrm{Im}(f)=\{f(x)\in A^{\prime}\mid x\in A\}
\end{equation*}
and
\begin{equation*}
\mathrm{Ker}(f)=\{x\in A\mid f(x)=0^{\prime}\}.
\end{equation*}
\end{definition}

In fact it is easy to show the following theorem.

\begin{theorem}\label{UP31}\cite{Iampan2014}
Let $A,B$ and $C$ be UP-algebras.
Then the following statements hold:
\begin{enumerate}[label=\textrm{(\arabic*)}]
\item\label{UP31_1} the identity mapping $I_{A}\colon A\to A$ is a UP-isomorphism,
\item\label{UP31_2} if $f\colon A\to B$ is a UP-isomorphism, then $f^{-1}\colon B\to A$ is a UP-isomorphism, and
\item\label{UP31_3} if $f\colon A\to B$ and $g\colon B\to C$ are UP-isomorphisms, then $g\circ f\colon A\to C$ is a UP-isomorphism.
\end{enumerate}
\end{theorem}


\begin{theorem}\label{UP23}\cite{Iampan2014}
Let $A$ be a UP-algebra and $B$ a UP-ideal of $A$.
Then the mapping $\pi_{B}\colon A\to A/\sim_{B}$ defined by $\pi_{B}(x)=(x)_{\sim_{B}}$ for all $x\in A$ is a UP-epimorphism, called the \textit{natural projection} from $A$ to $A/\sim_{B}$.
\end{theorem}


On a UP-algebra $A=(A,\cdot,0)$, we define a binary relation $\leq$ on $A$ as follows: for all $x,y\in A$,
\begin{equation}\label{eq1}
x\leq y \textrm{ if and only if } x\cdot y=0.
\end{equation}

\begin{proposition}\label{UP4}\cite{Iampan2014}
Let $A$ be a UP-algebra with a binary relation $\leq$ defined by \eqref{eq1}.
Then $(A,\leq)$ is a partially ordered set with $0$ as the greatest element.
\end{proposition}

We often call the partial ordering $\leq$ defined by \eqref{eq1} the \textit{UP-ordering} on $A$.
From now on, the symbol $\leq$ will be used to denote the UP-ordering, unless specified otherwise.


\begin{theorem}\label{UP24}\cite{Iampan2014}
Let $(A,\cdot,0_{A})$ and $(B,\ast,0_{B})$ be UP-algebras and let $f\colon A\to B$ be a UP-homomorphism.
Then the following statements hold:
\begin{enumerate}[label=\textrm{(\arabic*)}]
\item\label{UP24_1} $f(0_{A})=0_{B}$,
\item\label{UP24_8} for any $x,y\in A$, if $x\leq y$, then $f(x)\leq f(y)$,
\item\label{UP24_2} if $C$ is a UP-subalgebra of $A$, then the image $f(C)$ is a UP-subalgebra of $B$. In particular, $\mathrm{Im}(f)$ is a UP-subalgebra of $B$,
\item\label{UP24_5} if $D$ is a UP-subalgebra of $B$, then the inverse image $f^{-1}(D)$ is a UP-subalgebra of $A$. In particular, $\mathrm{Ker}(f)$ is a UP-subalgebra of $A$,
\item\label{UP24_6} if $C$ is a UP-ideal of $A$ such that $\mathrm{Ker}(f)\subseteq C$, then the image $f(C)$ is a UP-ideal of $f(A)$,
\item\label{UP24_7} if $D$ is a UP-ideal of $B$, then the inverse image $f^{-1}(D)$ is a UP-ideal of $A$. In particular, $\mathrm{Ker}(f)$ is a UP-ideal of $A$, and
\item\label{UP24_3} $\mathrm{Ker}(f)=\{0_{A}\}$ if and only if $f$ is injective.
\end{enumerate}
\end{theorem}



\section{Main Results}


In this section, we construct the fundamental theorem of UP-homomorphisms in UP-algebras.
We also give an application of the theorem to the first, second, third and fourth UP-isomorphism theorems in UP-algebras.


\begin{theorem}\label{ISO0} (Fundamental Theorem of UP-homomorphisms)
Let $(A,\cdot,0_{A})$ and $(B,\bullet,0_{B})$ be UP-algebras, and $f\colon A\to B$ a UP-homomorphism.
Then there exists uniquely a UP-homomorphism $\varphi$ from $A/\sim_{\mathrm{Ker}(f)}$ to $B$ such that $f=\varphi\circ\pi_{\mathrm{Ker}(f)}$. Moreover,
\begin{enumerate}[label=\textrm{(\arabic*)}]
\item\label{ISO0_1} $\pi_{\mathrm{Ker}(f)}$ is a UP-epimorphism and $\varphi$ a UP-monomorphism, and
\item\label{ISO0_2} $f$ is a UP-epimorphism if and only if $\varphi$ is a UP-isomorphism.
\end{enumerate}
As $f$ makes the following diagram commute,
\begin{displaymath}
\xymatrix{
A \ar[r]^f \ar[d]_{\pi_{\mathrm{Ker}(f)}} & B \\
A/\sim_{\mathrm{Ker}(f)} \ar[ur]_{\varphi} & }
\end{displaymath}
\end{theorem}

\begin{proof}
Put $K=\mathrm{Ker}(f)$.
By Theorem \ref{UP24} \ref{UP24_7}, we have $K$ is a UP-ideal of $A$.
It follows from Theorem \ref{UP19} \ref{UP19_2} that $(A/\sim_{K},\ast,(0_{A})_{\sim_{K}})$ is a UP-algebra.
Define
\begin{equation}\label{eqiso1}
\varphi\colon A/\sim_{K}\to B,(x)_{\sim_{K}}\mapsto f(x).
\end{equation}
Let $(x)_{\sim_{K}},(y)_{\sim_{K}}\in A/\sim_{K}$ be such that $(x)_{\sim_{K}}=(y)_{\sim_{K}}$.
Then $x\sim_{K}y$, so $x\cdot y\in K$ and $y\cdot x\in K$.
Thus
\begin{equation*}
f(x)\bullet f(y)=f(x\cdot y)=0_{B} \textrm{ and } f(y)\bullet f(x)=f(y\cdot x)=0_{B}.
\end{equation*}
By (UP-4), we have $f(x)=f(y)$ and so $\varphi((x)_{\sim_{K}})=\varphi((y)_{\sim_{K}})$.
Thus $\varphi$ is a mapping.
For any $x,y\in A$, we see that
\begin{equation*}
\varphi((x)_{\sim_{K}}\ast(y)_{\sim_{K}})=\varphi((x\cdot y)_{\sim_{K}})=f(x\cdot y)=f(x)\bullet f(y)=\varphi((x)_{\sim_{K}})\bullet\varphi((y)_{\sim_{K}}).
\end{equation*}
Thus $\varphi$ is a UP-homomorphism.
Also, since
\begin{equation*}
(\varphi\circ\pi_{K})(x)=\varphi(\pi_{K}(x))=\varphi((x)_{\sim_{K}})=f(x) \textrm{ for all } x\in A,
\end{equation*}
we obtain $f=\varphi\circ\pi_{K}$. We have shown the existence.
Let $\varphi^{\prime}$ be a mapping from $A/\sim_{K}$ to $B$ such that $f=\varphi^{\prime}\circ\pi_{K}$.
Then for any $(x)_{\sim_{K}}\in A/\sim_{K}$, we have
\begin{equation*}
\varphi^{\prime}((x)_{\sim_{K}})=\varphi^{\prime}(\pi_{K}(x))=(\varphi^{\prime}\circ\pi_{K})(x)=f(x)=(\varphi\circ\pi_{K})(x)=\varphi(\pi_{K}(x))=\varphi((x)_{\sim_{K}}).
\end{equation*}
Hence, $\varphi=\varphi^{\prime}$, showing the uniqueness.

\noindent\ref*{ISO0_1} By Theorem \ref{UP23}, we have $\pi_{K}$ is a UP-epimorphism.
Also, let $(x)_{\sim_{K}},(y)_{\sim_{K}}\in A/\sim_{K}$ be such that $\varphi((x)_{\sim_{K}})=\varphi((y)_{\sim_{K}})$.
Then $f(x)=f(y)$, and it follows from Proposition \ref{UP2} \ref{UP2_1} that
\begin{equation*}
f(x\cdot y)=f(x)\bullet f(y)=f(y)\bullet f(y)=0_{B},
\end{equation*}
that is, $x\cdot y\in K$.
Similarly, $y\cdot x\in K$.
Hence, $x\sim_{K}y$ and $(x)_{\sim_{K}}=(y)_{\sim_{K}}$.
Therefore, $\varphi$ a UP-monomorphism.

\noindent\ref*{ISO0_2} Assume that $f$ is a UP-epimorphism.
By \ref{ISO0_1}, it suffices to prove $\varphi$ is surjective.
Let $y\in B$. Then there exists $x\in A$ such that $f(x)=y$.
Thus $y=f(x)=\varphi((x)_{\sim_{K}})$, so $\varphi$ is surjective.
Hence, $\varphi$ is a UP-isomorphism.

Conversely, assume that $\varphi$ is a UP-isomorphism.
Then $\varphi$ is surjective.
Let $y\in B$. Then there exists $(x)_{\sim_{K}}\in A/\sim_{K}$ such that $\varphi((x)_{\sim_{K}})=y$.
Thus $f(x)=\varphi((x)_{\sim_{K}})=y$, so $f$ is surjective.
Hence, $f$ is a UP-epimorphism.
\end{proof}


\begin{theorem}\label{ISO1} (First UP-isomorphism Theorem)
Let $(A,\cdot,0_{A})$ and $(B,\bullet,0_{B})$ be UP-algebras, and $f\colon A\to B$ a UP-homomorphism.
Then
\begin{equation*}
A/\sim_{\mathrm{Ker}(f)}\cong\mathrm{Im}(f).
\end{equation*}
\end{theorem}

\begin{proof}
By Theorem \ref{UP24} \ref{UP24_2}, we have $\mathrm{Im}(f)$ is a UP-subalgebra of $B$.
Thus $f\colon A\to\mathrm{Im}(f)$ is a UP-epimorphism.
Applying Theorem \ref{ISO0} \ref{ISO0_2}, we obtain $A/\sim_{\mathrm{Ker}(f)}\cong\mathrm{Im}(f)$.
\end{proof}


\begin{lemma}\label{ISO4}
Let $(A,\cdot,0)$ be a UP-algebra, $H$ a UP-subalgebra of $A$, and $K$ a UP-ideal of $A$.
Denote $HK=\bigcup_{h\in H}(h)_{\sim_{K}}$.
Then $HK$ is a UP-subalgebra of $A$.
\end{lemma}

\begin{proof}
Clearly, $\emptyset\neq HK\subseteq A$.
Let $a,b\in HK$.
Then $a\in(x)_{\sim_{K}}$ and $b\in(y)_{\sim_{K}}$ for some $x,y\in H$, so $(a)_{\sim_{K}}=(x)_{\sim_{K}}$ and $(b)_{\sim_{K}}=(y)_{\sim_{K}}$.
Thus
\begin{equation*}
(a\cdot b)_{\sim_{K}}=(a)_{\sim_{K}}\ast(b)_{\sim_{K}}=(x)_{\sim_{K}}\ast(y)_{\sim_{K}}=(x\cdot y)_{\sim_{K}},
\end{equation*}
so $a\cdot b\in(x\cdot y)_{\sim_{K}}$.
Since $x,y\in H$, it follows from Proposition \ref{UP20} that $x\cdot y\in H$.
Thus $a\cdot b\in(x\cdot y)_{\sim_{K}}\subseteq HK$.
Hence, $HK$ is a UP-subalgebra of $A$.
\end{proof}


\begin{theorem}\label{ISO2} (Second UP-isomorphism Theorem)
Let $(A,\cdot,0)$ be a UP-algebra, $H$ a UP-subalgebra of $A$, and $K$ a UP-ideal of $A$.
Denote $HK/\sim_{K}=\{(x)_{\sim_{K}}\mid x\in HK\}$.
Then
\begin{equation*}
H/\sim_{H\cap K}\cong HK/\sim_{K}.
\end{equation*}
\end{theorem}

\begin{proof}
By Lemma \ref{ISO4}, we have $HK$ is a UP-subalgebra of $A$.
Then it is easy to check that $HK/\sim_{K}$ is a UP-subalgebra of $A/\sim_{K}$, thus $(HK/\sim_{K},\ast,(0)_{\sim_{K}})$ itself is a UP-algebra.
Also, it is obvious that $H\subseteq HK$, then
\begin{equation}\label{eqiso2}
f\colon H\to HK/\sim_{K},x\mapsto(x)_{\sim_{K}},
\end{equation}
is a mapping.
For any $x,y\in H$, we have
\begin{equation*}
f(x\cdot y)=(x\cdot y)_{\sim_{K}}=(x)_{\sim_{K}}\ast(y)_{\sim_{K}}=f(x)\ast f(y).
\end{equation*}
Thus $f$ is a UP-homomorphism.
We shall show that $f$ is a UP-epimorphism with $\mathrm{Ker}(f)=H\cap K$.
For any $(x)_{\sim_{K}}\in HK/\sim_{K}$, we have $x\in HK=\bigcup_{h\in H}(h)_{\sim_{K}}$.
Then there exists $h\in H$ such that $x\in(h)_{\sim_{K}}$ and so $(x)_{\sim_{K}}=(h)_{\sim_{K}}$.
Thus $f(h)=(h)_{\sim_{K}}=(x)_{\sim_{K}}$.
Therefore, $f$ is a UP-epimorphism.
Also, for any $h\in H$, if $h\in\mathrm{Ker}(f)$, then $f(h)=(0)_{\sim_{K}}$.
Since $f(h)=(h)_{\sim_{K}}$, we obtain $(h)_{\sim_{K}}=(0)_{\sim_{K}}$.
By (UP-2) and \eqref{eq9}, we have $h=0\cdot h\in K$.
Thus $h\in H\cap K$, that is, $\mathrm{Ker}(f)\subseteq H\cap K$.
On the other hand, if $h\in H\cap K$, by $h\in H$, $f(h)$ is well-defined, by $h\in K$ and $0\in K$, $h\cdot 0\in K$ and $0\cdot h\in K$.
By \eqref{eq9}, we have $h\sim_{K} 0$ and so $(h)_{\sim_{K}}=(0)_{\sim_{K}}$.
Thus $f(h)=(h)_{\sim_{K}}=(0)_{\sim_{K}}$.
So, $h\in\mathrm{Ker}(f)$, that is, $H\cap K\subseteq\mathrm{Ker}(f)$.
Therefore, $\mathrm{Ker}(f)=H\cap K$.
Now, Theorem \ref{ISO1} gives $H/\sim_{H\cap K}\cong HK/\sim_{K}$.
\end{proof}


\begin{theorem}\label{ISO3} (Third UP-isomorphism Theorem)
Let $(A,\cdot,0)$ be a UP-algebra, and $H$ and $K$ UP-ideals of $A$ with $H\subseteq K$.
Then
\begin{equation*}
(A/\sim_{H})/\sim_{(K/\sim_{H})}\cong A/\sim_{K}.
\end{equation*}
\end{theorem}

\begin{proof}
By Theorem \ref{UP19} \ref{UP19_2}, we obtain $(A/\sim_{K},\ast,(0)_{\sim_{K}})$ and $(A/\sim_{H},\ast^{\prime},(0)_{\sim_{H}})$ are UP-algebras.
Define
\begin{equation}\label{eqiso3}
f\colon A/\sim_{H}\to A/\sim_{K},(x)_{\sim_{H}}\mapsto(x)_{\sim_{K}}.
\end{equation}
For any $x,y\in A$, if $(x)_{\sim_{H}}=(y)_{\sim_{H}}$, then $x\cdot y,y\cdot x\in H$.
Since $H\subseteq K$, we obtain $x\cdot y,y\cdot x\in K$.
Thus $(x)_{\sim_{K}}=(y)_{\sim_{K}}$, so $f((x)_{\sim_{H}})=f((y)_{\sim_{H}})$.
Thus $f$ is a mapping.
Also, for any $x,y\in A$, we see that
\begin{equation*}
f((x)_{\sim_{H}}\ast^{\prime}(y)_{\sim_{H}})=f((x\cdot y)_{\sim_{H}})=(x\cdot y)_{\sim_{K}}=(x)_{\sim_{K}}\ast(y)_{\sim_{K}}=f((x)_{\sim_{H}})\ast f((y)_{\sim_{H}}).
\end{equation*}
Thus $f$ is a UP-homomorphism.
Clearly, $f$ is surjective.
Hence, $f$ is a UP-epimorphism.
We shall show that $\mathrm{Ker}(f)=K/\sim_{H}$.
In fact,
\begin{align*}
   \mathrm{Ker}(f)  & = \{(x)_{\sim_{H}}\in A/\sim_{H}\mid f((x)_{\sim_{H}})=(0)_{\sim_{K}}\} \\
        & = \{(x)_{\sim_{H}}\in A/\sim_{H}\mid (x)_{\sim_{K}}=(0)_{\sim_{K}}\}  \\
        & = \{(x)_{\sim_{H}}\in A/\sim_{H}\mid x=0\cdot x\in K\} \tag{By (UP-2)} \\
        & = K/\sim_{H}.
\end{align*}
Now, Theorem \ref{ISO1} gives $(A/\sim_{H})/\sim_{(K/\sim_{H})}\cong A/\sim_{K}$.
\end{proof}


\begin{theorem}\label{ISO5} (Fourth UP-isomorphism Theorem)
Let $(A,\cdot,0_{A})$ and $(B,\bullet,0_{B})$ be UP-algebras, and $f\colon A\to B$ a UP-epimorphism.
Denote $\mathcal{A}=\{X\mid X$ is a UP-ideal of $A$ containing $\mathrm{Ker}(f)\}$ and $\mathcal{B}=\{Y\mid Y$ is a UP-ideal of $B\}$.
Then the following statements hold:
\begin{enumerate}[label=\textrm{(\arabic*)}]
\item\label{ISO5_1} there is an inclusion preserving bijection
\begin{equation}\label{eqiso4}
\varphi\colon \mathcal{A}\to\mathcal{B},X\mapsto f(X),
\end{equation}
with inverse given by $Y\mapsto f^{-1}(Y)$, and
\item\label{ISO5_2} for any $X\in\mathcal{A}$,
\begin{equation*}
A/\sim_{X}\cong B/\sim_{f(X)}.
\end{equation*}
\end{enumerate}
\end{theorem}

\begin{proof}
\noindent\ref*{ISO5_1} For any $X\in\mathcal{A}$, it follows from Theorem \ref{UP24} \ref{UP24_6} that $f(X)$ is a unique UP-ideal of $B$ such that $\varphi(X)=f(X)$.
Thus $\varphi$ is a mapping.
For any $X_{1},X_{2}\in\mathcal{A}$, if $\varphi(X_{1})=\varphi(X_{2})$, then $f(X_{1})=f(X_{2})$.
Since $\mathrm{Ker}(f)\subseteq X_{1}$, we obtain $X_{1}=f^{-1}(f(X_{1}))$.
Indeed, let $x\in f^{-1}(f(X_{1}))$. Then $f(x)\in f(X_{1})$, so $f(x)=f(x_{1})$ for some $x_{1}\in X_{1}$.
Applying Proposition \ref{UP2} \ref{UP2_1}, we have $f(x_{1}\cdot x)=f(x_{1})\bullet f(x)=f(x_{1})\bullet f(x_{1})=0_{B}$.
Thus $x_{1}\cdot x\in\mathrm{Ker}(f)\subseteq X_{1}$, it follows from Theorem \ref{UP9} \ref{UP9_1} that $x\in X_{1}$.
So, $f^{-1}(f(X_{1}))\subseteq X_{1}$.
Clearly, $X_{1}\subseteq f^{-1}(f(X_{1}))$.
Similarly, since $\mathrm{Ker}(f)\subseteq X_{2}$, we obtain $X_{2}=f^{-1}(f(X_{2}))$.
Thus $X_{1}=f^{-1}(f(X_{1}))=f^{-1}(f(X_{2}))=X_{2}$.
Hence, $\varphi$ is injective.
Also, for any $Y\in\mathcal{B}$, we obtain $Y= f(f^{-1}(Y))$ because $f$ is surjective.
Applying Theorem \ref{UP24} \ref{UP24_7}, we have $f^{-1}(Y)$ is a UP-ideal of $A$ with $\mathrm{Ker}(f)\subseteq f^{-1}(Y)$.
Thus $f^{-1}(Y)\in\mathcal{A}$ is such that $\varphi(f^{-1}(Y))=f(f^{-1}(Y))=Y$.
Hence, $\varphi$ is surjective.
Therefore, $\varphi$ is bijective.
Finally, for any $Y\in\mathcal{B}$, we get that $Y=\varphi(f^{-1}(Y))$.
Hence, $\varphi^{-1}(Y)=f^{-1}(Y)$.

\noindent\ref*{ISO5_2} By Theorem \ref{UP24} \ref{UP24_6} and Theorem \ref{UP19} \ref{UP19_2}, we have $f(X)$ is a UP-ideal of $B$ and $(B/\sim_{f(X)},\ast,(0_{B})_{\sim_{f(X)}})$ is a UP-algebra.
It follows from Theorem \ref{UP23} that $\pi_{f(X)}\colon B\to B/\sim_{f(X)}$ is a UP-epimorphism.
Thus $\pi_{f(X)}\circ f\colon A\to B/\sim_{f(X)}$ is a UP-epimorphism.
We shall show that $\mathrm{Ker}(\pi_{f(X)}\circ f)=X$.
In fact,
\begin{align*}
   \mathrm{Ker}(\pi_{f(X)}\circ f)  & = \{a\in A\mid (\pi_{f(X)}\circ f)(a)=(0_{B})_{\sim_{f(X)}}\} \\
        & = \{a\in A\mid \pi_{f(X)}(f(a))=(0_{B})_{\sim_{f(X)}}\} \\
        & = \{a\in A\mid (f(a))_{\sim_{f(X)}}=(0_{B})_{\sim_{f(X)}}\} \\
        & = \{a\in A\mid f(a)=0_{B}\bullet f(a)\in f(X)\} \tag{By (UP-2)} \\
        & = f^{-1}(f(X)) \\
        & = X.
\end{align*}
Applying Theorem \ref{ISO1}, we have $A/\sim_{X}\cong B/\sim_{f(X)}$.
\end{proof}



\section*{Acknowledgment}
The author wish to express their sincere thanks to the referees for the valuable suggestions which lead to an improvement of this paper.


{\bf Received: \today}

\end{document}